\documentclass{amsproc}

\usepackage{amsmath}
\usepackage{amssymb}
\usepackage[mathcal]{eucal}
\usepackage[bb=pazo]{mathalfa}
\usepackage{rotating}
\usepackage{graphicx}
\usepackage[tableposition=below]{caption}
\usepackage{pigpen}
\usepackage{multirow}

\DeclareGraphicsExtensions{.png,.pdf,.eps}
\usepackage[all,2cell,cmtip]{xy}
\usepackage{tikz}
\usepackage{color}
\usepackage{longtable}
\usepackage{soul}
\newtheorem{theorem}{Theorem}[section]
\newtheorem*{theorem*}{Theorem}
\newtheorem*{corollary*}{Corollary}
\newtheorem{lemma}[theorem]{Lemma}
\newtheorem{corollary}[theorem]{Corollary}
\newtheorem{proposition}[theorem]{Proposition}

\theoremstyle{definition}

\newtheorem{definition}[theorem]{Definition}

\newtheorem{remark}[theorem]{Remark}

\newtheorem{setup}[theorem]{Setup}
\newtheorem{example}[theorem]{Example}

\def\C{{\mathbb C}}

\def\P{{\mathbb P}}

\def\R{{\mathbb R}}
\def\Z{{\mathbb Z}}

\def\cD{{\mathcal D}}
\def\cE{{\mathcal E}}
\def\cF{{\mathcal F}}
\def\cG{{\mathcal{G}}}

\def\cM{{\mathcal M}}

\def\cO{{\mathcal{O}}}

\def\cU{{\mathcal U}}
\def\cV{{\mathcal V}}
\def\cW{{\mathcal W}}

\def\fn{{\mathfrak n}}
\def\fg{{\mathfrak g}}
\def\fh{{\mathfrak h}}

\def\GL{\operatorname{Gl}}
\def\PGL{\operatorname{PGl}}

\def\DA{{\rm A}}
\def\DB{{\rm B}}
\def\DC{{\rm C}}
\def\DD{{\rm D}}
\def\DE{{\rm E}}
\def\DF{{\rm F}}
\def\DG{{\rm G}}

\def\lra{\longrightarrow}

\def\ra{\rightarrow}
\def\lra{\longrightarrow}
\def\rat{\dashrightarrow}

\def\Mo{\operatorname{\hspace{0cm}M}}
\def\Na{\operatorname{\hspace{0cm}N}}

\def\operatorname#1{\mathop{\rm #1}\nolimits}

\def\Aut{\operatorname{Aut}}

\def\Chow{\operatorname{Chow}}

\def\Hom{\operatorname{Hom}}

\def\Pic{\operatorname{Pic}}

\def\rank{\operatorname{rank}}
\def\rk{\operatorname{rk}}

\def\rat{\operatorname{RatCurves}}
\def\cdim{\operatorname{cdim}}

\def\NE{{\operatorname{NE}}}

\def\Nu{{\operatorname{N_1}}}

\def\Ad{\operatorname{Ad}}

\def\Ht{\operatorname{ht}}

\newcommand{\ol}[1]{\overline{#1}}
\newcommand{\pb}{\ar@{}[dr]|(.50){\text{\pigpenfont J}}}

\makeatletter
 
\newcommand*\wthelper[2]{%
        \hbox{\dimen@\accentfontxheight#1%
                \accentfontxheight#11.15\dimen@
                $\m@th#1\widetilde{#2}$%
                \accentfontxheight#1\dimen@
        }%
}

\newcommand*\accentfontxheight[1]{%
        \fontdimen5\ifx#1\displaystyle
                \textfont
        \else\ifx#1\textstyle
                \textfont
        \else\ifx#1\scriptstyle
                \scriptfont
        \else
                \scriptscriptfont
        \fi\fi\fi3
}
\makeatother

\newcommand{\shse}[3]{0 ~\ra ~#1~ \lra ~#2~ \lra ~#3~ \ra~ 0}

\makeindex

\begin{document}

\title[On uniform flag bundles on Fano manifolds]{On uniform flag bundles on Fano manifolds}

\author[R. Mu\~noz]{Roberto Mu\~noz}

\address{Departamento de Matem\'atica Aplicada, ESCET, Universidad
Rey Juan Carlos, 28933-M\'ostoles, Madrid, Spain}
\email{roberto.munoz@urjc.es}
\thanks{First and third author partially supported by the spanish government project MTM2015-65968-P. Second author supported by PRIN project ``Geometria delle variet\`a algebriche''. Second and third author supported by the Department of Mathematics of the University of Trento.}

\author[G. Occhetta]{Gianluca Occhetta}
\address{Dipartimento di Matematica, Universit\`a di Trento, via
Sommarive 14 I-38123 Povo di Trento (TN), Italy} 
\email{gianluca.occhetta@unitn.it}

\author[L.E. Sol\'a Conde]{Luis E. Sol\'a Conde}
\address{Dipartimento di Matematica, Universit\`a di Trento, via
Sommarive 14 I-38123 Povo di Trento (TN), Italy}
\email{lesolac@gmail.com}

\begin{abstract}
As a natural extension of the theory of uniform vector bundles on Fano manifolds, we consider  uniform principal bundles, and study them by means of the associated flag bundles, as their natural projective geometric realizations. In this paper we develop the necessary background, and prove some theorems that are flag bundle counterparts of some of the central results in the theory of uniform vector bundles.
\end{abstract}
\maketitle

\section{Introduction}\label{sec:intro}

Although the fact that a vector bundle over the complex projective line $\P^1$ splits as a direct sum of line bundles is a theorem whose history goes back to the end of the nineteenth century, it was not until the 1950's that it achieved its modern form as posed by Grothendieck. Working upon the ideas developed previously by the french school of Cartan and Borel, he considered vector bundles as geometric realizations of principal $G$-bundles, with $G$ reductive, via certain representations of the group $G$, and showed that every principal bundle over $\P^1$ is determined 
by a co-character of a Cartan subgroup of $G$. It is then this discrete invariant that determines the splitting type of any vector bundle associated with the principal bundle via a given linear representation of $G$. 

For varieties different from $\P^1$ the situation is far more complicated, since, even in the case of  non rational curves, vector bundles are not, in general, determined by discrete invariants. Moreover, most varieties admit bundles that are not direct sums; perhaps the simplest example of this kind is the tangent bundle over the projective plane $\P^2$, which also shows that decomposability of vector bundles fails already within the class of homogeneous vector bundles. 

On the other hand, the work of Grauert, Schwarzenberger and Van de Ven in the 1960's and 70's, \cite{Schw,VdV}, leads to the conclusion that a natural way to extend Grothendieck's theorem is to consider varieties covered by a family of rational curves, and bundles whose restrictions to all the curves of the family are isomorphic. These bundles are called {\em uniform}; note that by means of Grothendieck's theorem this condition can be written in terms of co-characters or splitting types.  Grauert et al. considered only the particular case of lines in the projective space, posing the problem of understanding whether every bundle on $\P^n$ uniform with respect to the family of lines should be homogeneous. Remarkably, it was shown that the question has an affirmative answer for low rank bundles (\cite{Sato, VdV, Elenc, EHS, Ellia, Ballico1}). 

Besides their relation with homogeneity, the concepts of splitting type and uniformity have been extensively used within the theory of vector bundles, particularly in the case of $\P^n$  
--we refer to \cite{OSS} for a complete account on the topic. Among other significant results, we should mention here the characterization of trivial bundles as the only uniform bundles with trivial splitting type (see \cite{VdV}, \cite[Theorem 3.2.1]{OSS}), and Grauert--M\"ulich theorem, stating that for a semistable bundle on $\P^n$ (not necessarily uniform) the gaps in the general splitting type cannot be greater than one (\cite{GM},  \cite[Theorem 2.1.4]{OSS}).   

On the other hand, we have already mentioned that the concept of uniformity makes perfect sense for vector bundles on varieties covered by a family of rational curves; in particular, the problem of determining its relation with homogeneity can be posed for any rational homogeneous space, and has been already considered  in the case of quadrics \cite{Ballico2}, Grassmannians \cite{Guyot}, and some other varieties \cite{MOS1, Wis2}. Furthermore, Grothendieck's theorem allows us to classify not only vector bundles over $\P^1$, but also principal bundles, hence one may extend the concept of uniformity to the setting of principal bundles on rationally connected varieties, and study its relation with homogeneity in the case in which the base is homogeneous. The results we have obtained suggest that many classical statements on $\P^n$ could be extended to this setting. 

In order to study principal $G$-bundles, we will make use of some of their projective geometric realizations. Instead of considering the projectivizations of some of their associated vector bundles, via representations of $G$, we have chosen to work directly with their associated {\em flag bundles}, which are constructed upon the action of the defining group $G$ on the flag manifold $G/B$. A geometrical reason for this choice is that flag manifolds are particularly simple when one looks at their families of minimal rational curves; in fact, the intersection properties of these families contain the necessary information to reconstruct the action of the group $G$ (see \cite{OSWW,OSW}). As a second motivation, flag bundles can be constructed upon rational homogeneous bundles, which appear sometimes within the framework of the theory of Fano manifolds (see \cite{Hw,LM}).

The goal of the present paper, which is the first of a project in which we study uniform principal $G$-bundles ($G$ semisimple)  over Fano manifolds, is to develop the background necessary to address these questions, and to present a number of theorems that are flag bundle versions of some of the central results in the theory of uniform vector bundles.  
\subsection{Outline}
We start in Section \ref{sec:prelim} with some generalities on flag bundles and their relation to principal bundles and their representations. We pay special attention to certain filtrations of their relative tangent bundles, that will be useful later on. In Section \ref{sec:diagred} we define decomposability, reducibility and diagonalizability for flag bundles, generalizing the different standard decomposability notions of vector bundles. We also discuss the interactions among these concepts, relating them with the existence of sections of the associated rational homogeneous bundles (see Corollary \ref{cor:decompo1} and Section \ref{ssec:redu}). For instance, as a generalization of the fact that a $\P^1$-bundle on a Fano manifold of Picard number one is decomposable if and only if it admits a section, we show the following (cf. Corollary \ref{cor:decompo1}):

\begin{theorem*}
Let $X$ be a Fano manifold of Picard number one, and $\pi:Y\to X$ a  flag bundle. Then $\pi$ is diagonalizable if and only if it admits a section.
\end{theorem*}

Section \ref{sec:groth} is devoted to the study of flag bundles over the projective line. In particular we recall the concept of {\em tag} of a $G/B$-bundle on a rational curve, defined in \cite{OSW} as a $\rk(G)$-vector of non-negative integers associated with the nodes of the Dynkin diagram of $G$,  which is a geometric counterpart of Grothendieck's classifying co-character (see Remark \ref{rem:tagsplit} for the relation between the tag and the splitting type when $G/B$ is the complete flag manifold of a projective space). Furthermore, we describe the Mori cones (Proposition \ref{prop:moricone}) and the families of  minimal sections (Proposition \ref{prop:trivialsub}) of $G/B$-bundles over $\P^1$.

 The concept of tag allows us to define the notion of uniform flag bundle, that we  introduce in Section \ref{sec:unif}.
After discussing families of minimal sections of uniform flag bundles, we  prove  a characterization of trivial flag bundles in terms of their tags with respect to certain families of rational curves, namely Theorem \ref{thm:trivial1}: 
\begin{theorem*}\label{thm:trivial2}
Let $X$ be a manifold, rationally chain connected with respect to unsplit families of rational curves $\cM_1, \dots,\cM_s$, and $\pi:Y\to X$ a $G/B$-bundle over $X$. 
Then $ Y\cong X\times G/B$ is trivial as a $G/B$-bundle over $X$ if and only if for every rational curve $\Gamma_i$ belonging to a family $\cM_i$, the tag of the $G/B$-bundle $Y$ on $\Gamma_i$ is zero.
\end{theorem*}
Our proof is complementary to the different proofs characterizing trivial vector bundles in terms of their restrictions to curves (\cite{AW, BdS,Pan}, see also \cite[Proposition 2.4]{kyoto}), stressing the interplay between both approaches --flag bundle versus vector bundle-- to these questions.

The last section is devoted to the study of criteria of reducibility and diagonalizability  for uniform flag bundles. In this case it is convenient to consider a special type of reducibility, named {\it uniform reducibility}, defined in Section \ref{ssec:redu} upon the particular family of rational curves with respect to which uniformity is defined. For instance, in Lemma \ref{lem:irred1} we show that for a uniformly reducible flag bundle, its diagonalizability reduces to the diagonalizability of an auxiliary flag bundle of lower rank. Roughly speaking, one could say that uniform flag bundles whose tag is sufficiently positive are reducible. In this spirit, we obtain a flag bundle counterpart of the well known Grauert--M\"ulich theorem (see Theorem \ref{thm:GMcrit1} for a precise statement):
\begin{theorem*}\label{thm:GMcrit1}
Any irreducible uniform flag bundle on a Fano manifold of Picard number one contains only $0$'s and $1$'s in its tag.
\end{theorem*} 
In particular the problem of diagonalizability of uniform flag bundles of low rank reduces to that of flag bundles tagged with zeros and ones. Furthermore, not every tag with ones and zeros may occur on a non diagonalizable bundle: for every node $j$ marked by $1$ we define an integer $m_j$ which depends on the number of nodes marked with zero ``adjacent'' to $j$ (for the precise definition see Table \ref{tab:mj}) and we set $m:=\max{m_j}+1$. Then

\begin{theorem*}\label{thm:some1}
Let $X$ be a Fano manifold, $\cM$ be an unsplit dominating complete family of rational curves, whose evaluation morphism $q:\cU\to X$ has connected fibers.  
Let $\pi:Y\to X$ be a uniform $G/B$-bundle over $X$, with tag $(d_1,\dots,d_n)$. 
If every morphism  $f:\cU\to \cU'$ over $X$ whose image has relative dimension smaller than $m$ is relatively constant then $\pi$ is diagonalizable.
\end{theorem*}

In a nutshell, the above statement (see Theorem \ref{thm:some1}) tells us that, for a non diagonalizable bundle, every $1$ in the tag must be conveniently isolated by zeros, depending 
on the geometry of rational curves on the variety.

We finish the section and the paper with two applications of this result, Corollaries \ref{prop:only1} and \ref{cor:only1}, in which we consider the case of a tag with no zeroes.


\section{Setup and preliminaries}\label{sec:prelim}

Along this paper $X$ will denote a smooth complex projective algebraic variety. A {\it $Z$-bundle} over $X$ is a smooth morphism $\pi:Y\to X$ whose scheme theoretical fibers are isomorphic to $Z$. A well known theorem of Fischer and Grauert states that such a bundle is locally trivial in the analytic topology: there exists an analytic open covering $\{\cV_i,\,\, i \in I\}$ of $X$ and trivializations $\phi_i : \cV_i \times Z \to \pi^{-1}(\cV_i)$ (isomorphisms commuting with the corresponding projections onto $\cV_i$). If we consider $\{\phi_i,\,\,i\in I\}$ to be the whole atlas of trivializations of $\pi$, then $\pi$ is completely determined by the transitions $\phi_{ij}:=\phi_i^{-1}\circ\phi_j:(\cV_{i}\cap \cV_{j})\times Z\to (\cV_{i}\cap \cV_{j})\times Z$, that can be thought of as maps $\theta_{ij}:\cV_{ij}:=\cV_{i}\cap \cV_{j}\to \Aut(Z)$. Then $\pi$ is determined by the corresponding cocycle $\theta=(\theta_{ij})\in H^1(X,\Aut(Z))$ (by abuse of notation, we mean here the cohomology of the sheafified group $\Aut(Z)$ on the analytic space associated with $X$). Furthermore, in the case in which $X$ is simply connected -- later our base varieties will be Fano manifolds --   the defining cocycle of a $Z$-bundle takes values in the identity component $\Aut(Z)^\circ$ of $\Aut(Z)$.

We will be mostly interested in {\it $G/B$-bundles}, that is in the case in which $Z$ is of the form $G/B$, where $G$ is a semisimple complex algebraic group, with Dynkin diagram $\cD$, and $B\subset G$ is a Borel subgroup; we will usually refer to them simply as {\em flag bundles}. Assuming that $X$ is simply connected, $\pi$ is determined by a cocycle $\theta\in H^1(X,\Aut(G/B)^\circ)$. Moreover, it is well known that $\Aut(G/B)^\circ$ is semisimple, isogenous to $G$ (see \cite[Remark 2.1]{OSW}), and that we may write $G/B$ as a quotient of $\Aut(G/B)^\circ$ by a Borel subgroup. We will always assume that the group $G$ from which $G/B$ is obtained as a quotient is $\Aut(G/B)^\circ$, and then we may say that a $G/B$-bundle on a simply connected manifold $X$ is completely determined by a cocycle $\theta\in H^1(X,G)$. Alternatively the $G/B$-bundle may be reconstructed from this cocycle by means of the Borel construction: $\theta$ determines a $G$-principal bundle $\pi_G:E\to X$, and we may identify $Y$ with the algebraic variety $$E\times^GG/B:=(E\times G/B)/\sim,\qquad (e,gB)\sim(eh,h^{-1}gB),\,\,\forall h\in G,$$ and then $\pi$ corresponds to the natural map sending the class of $(e,gB)$ to $\pi_G(e)$.

If we consider a maximal torus $H\subset B$, it determines a root system $\Phi$, contained in the Lie algebra $\fh$ of $H$, whose Weyl group $W$ is isomorphic to the quotient $N(H)/H$ of the normalizer $N(H)$ of $H$ in $G$. Within $\Phi$, $B$ determines a base of positive simple roots $\Delta$. 
 We will always choose an ordering of the set of simple roots $\Delta=\{\alpha_1,\dots,\alpha_n\}$, (in the case of simple algebraic group we will always choose the ordering of \cite[p.~58]{Hum2}), and denote by $D$ the set of indices $\{1,\dots,n\}$. By definition, the {\it rank} of the semisimple group $G$ is defined as $\rk(G):=\dim H=\sharp(\Delta)=n$. 

We denote by $r_i$ the reflection associated with $\alpha_i$. Then, for every subset $I\subset D$, we may consider a parabolic subgroup $P(I)$ defined by $P(I):=BW_IB$, where $W_I\subset W$ is the subgroup of $W$ generated by the reflections $r_i$, $i\in I$.  
Going back to our setting, for every such subset $I\subset D $  there is a factorization:
\begin{equation}\label{eq:contract}
\xymatrix{Y\ar@/^1pc/[rr]^{\pi}\ar[r]_{\rho_I}&Y_I\ar[r]_{\pi_I}&X}
\end{equation}
where $Y_I:=E\times^GG/P(I)$  is a {\em $G/P(I)$-bundle} over $X$. In the case in which $I=\{i\}$, we will simply write $\rho_i, \pi_i, Y_i$. 

Finally, we denote by $N^1(Y|X)$ the cokernel of the pullback map $N^1(X)\to N^1(Y)$, between the real vector spaces of classes of $\R$-divisors in $X$ and $Y$.  It is a vector space of dimension equal to the Picard number of $G/B$ (which is equal to $n$), that we may (and will) identify with the linear subspace of $N^1(Y)$ generated by the linearly independent set $\{-K_i,\,\,i\in D\}$, where $-K_i$ denotes the relative anticanonical divisor of the elementary contraction $\rho_{i}:Y\to Y_i$, for any $i\in D$. 
In particular, we may identify $N^1(Y|X)$ with the real vector space $\R\Phi\subset\fh$, by sending $-K_i$ to $\alpha_i$, for all $i\in D$. In this way, denoting by $\Gamma_i$ the numerical class of a fiber of the contraction $\rho_i$, for every $i$, the Cartan matrix of $G$ is the matrix of intersection numbers $(-K_i\cdot\Gamma_j)$. In the sequel, we will always think of the roots of $G$ as the corresponding integral combinations of divisors $-K_i$.   

\subsection{Standard constructions}\label{ssec:stconst}
We include here some classical constructions with principal and fiber bundles.

\noindent{\it 1. Pullback.} Given a $Z$-bundle $\pi:Y\to X$, and a morphism $f:X'\to X$, the fiber product $f^*Y:=Y \times_X X'$ has a natural structure of $Z$-bundle over $X'$. In the case in which $Z=G/B$ and $\pi$ is determined by a cocycle $\theta\in H^1(X,G)$, the bundle $Y\times_X X'\to X'$ corresponds to the image of $\theta$ by the pullback map $H^1(X,G)\to H^1(X',G)$. 
 
\noindent{\it 2. Extension.} Given any morphism of Lie groups $G\to G'$ ($G'$ semisimple), there is a natural map $H^1(X,G)\to H^1(X,G')$ that sends $\theta$ to a cocycle $\theta'$ defining a $G'/B'$-bundle, which may also be described as the $G'/B'$-bundle $E\times^GG'/B'$, where we consider the action of $G$ on $G'/B'$ induced by the map $G\to G'$. 

\noindent{\it 3. Reduction.} Conversely, if $\theta\in H^1(X,G)$ defines a $G/B$-bundle $\pi:Y\to X$, and $f:G'\to G$ is a homomorphism of algebraic groups, we say that $\theta$ admits a {\it reduction to $G'$} if $\theta$ lies in the image of the natural map $H^1(X,G')\to H^1(X,G)$. In the case in which the map $f$ is the inclusion $P:=P(I) \hookrightarrow G$, the reduction to $P$ is equivalent to the existence of a section $\sigma_I$ of $\pi_I:Y_I\to X$. In particular, in the case $P=B$, the reduction of $\theta$ to $B$ is equivalent to the existence of a section of $\pi$. 
Moreover, considering the {\it semisimple part $G_P$ of} $P$ (which is, by definition, the quotient of $P$ by its unipotent subgroup, and then by the center of the image), and its Borel subgroup $B_P:=BG_P\subset G_P$, the extension of $\theta$ to $G_P$ defines a  $G_P/B_P$-bundle $\pi^P:Y^P\to X$. Furthermore, by construction, $Y^P$ admits an embedding $i_P$ into $Y$ satisfying that $\sigma_I\circ\pi^P=\pi_I\circ i_P$.

\noindent{\it 4. Product.} Given two semisimple groups, $G,G'$, and two flag bundles $\pi: Y\to X$, $\pi': Y'\to X$, determined by cocycles $\theta\in H^1(X,G)$ and $\theta'\in H^1(X,G')$, and given any morphism $\rho: G\times G'\to G''$, the cocycle $\rho(\theta,\theta')\in H^1(X,G'')$ defines a flag bundle over $X$. Even in the case in which $\rho$ is injective, the flag bundle obtained is not, in general, the fiber product $Y\times_XY'$.

\subsection{Vector and projective bundles associated with representations}\label{ssec:bunrep}

An especially important case of extension appears when one considers rational or projective representations of the group $G$, that is homomorphisms of algebraic groups from $G$ to $\GL(V)$, or $\PGL(V)$, where $V$ is a finite dimensional complex vector space. These representations give rise to vector bundles $E\times^GV$, and to projective bundles $E\times^G\P(V)$, respectively. 

It is well known that rational representations of $G$ are determined by certain sets of elements in the {\em lattice of characters} of $H$, $$\Mo(H):=\Hom(H,\C^*);$$ 
more concretely, we may consider the induced action of the Weyl group $W$ on the vector space $\Mo(H)\otimes_\Z\R$, and the {\em fundamental Weyl chamber} $\cW\subset \Mo(H)\otimes_\Z\R$ determined by the base $\Delta$. Then an irreducible rational representation of $G$ is determined by a {\em highest weight}, that is, an element $w_{\max}\in \cW\cap \Mo(H)$. The induced action of $H$ on $V$  provides a decomposition 
$$
V=\bigoplus_{w\in \Mo(H)}V_w,
$$  
and the elements $w\in \Mo(H)$ for which $V_w\neq 0$ (called the {\em weights of the representation}) are the points in the intersection of the convex hull of of the orbit $W.w_{\max}$ with $\Mo(H)$ for which the difference with $w_{\max}$ belongs to the lattice generated by $\Delta$ (see for instance \cite[Lecture~14]{FuHa}). If $V$ is not irreducible, it admits a filtration on subrepresentations of $G$, whose quotients are irreducible, each of them having a description as above.

On the other hand, given a projective representation $\P(V)$ of $G$, we cannot claim that it is the projectivization of a rational representation $V$ of the group $G$, but of another semisimple group $G'$ isogenous to $G$. In general, the weights of any such $V$ lie in the lattice contained in $\Mo(H)\otimes_\Z\R$ generated by the fundamental weights $\{\lambda_1,\dots,\lambda_n\}$ of the Lie algebra $\fg$, defined by the property:
$$
r_i(\lambda_j)=\left\{\begin{array}{ll}\lambda_i-\alpha_i&\mbox{if }j=i,\\\lambda_j&\mbox{otherwise.}\end{array}\right.
$$

Note that given a $G/B$-bundle $\pi:Y\to X$, and a set of indices $I\subset D$, the variety $G/P(I)$ can be seen as an orbit of the action of 
$G$ on a projective space $\P(V_I)$ (see \cite[Claim 23.52]{FuHa}), and the above construction provides and embedding of $Y_I=E\times^GG/P(I)$ into the projective bundle $E\times^G\P(V_I)$.

\subsection{Filtrations of the relative tangent bundle}\label{ssec:filt}

Let $\pi:Y \to X$ be a $G/B$-bundle, and denote by $m$ the dimension of $G/B$, which equals the cardinality of  $\Phi^+\subset\Phi$, which is the set of roots that are nonnegative linear combinations of elements of the base $\Delta$. A total ordering $(L_{1},L_2,\dots,L_{m})$ of the elements of $\Phi^+$ is called {\it admissible} if, for every $L_j,L_{j'}, L_{j''}\in\Phi^+$,  $L_j+L_{j'}=L_{j''}$ implies that $j,j'<j''$. 
Note that, for instance, any total ordering of $\Phi^+$ satisfying $\Ht(L_j)\leq \Ht(L_{j+1})$ (where the {\it height} of a positive root is defined as $\Ht(L):=\sum_ia_i$ for $L=-\sum_ia_iK_i$) is obviously admissible.
 
Then, following \cite{MOSW}, for every admissible order we may construct a filtration of the relative tangent bundle $T_{Y|X}$:
$$
0=\cE_{m}\subset \cE_{m-1}  
\subset\dots\subset\cE_{0}=T_{Y|X},
$$
whose quotients satisfy:
$$
\cE_{r}/\cE_{r+1}\cong\cO_Y(L_{m-r}),\quad \mbox{for }r\in\{0,\dots,m-1\}.
$$
In particular we have (see \cite[Lemma 2.2]{OSW} for an explicit formula):
\begin{lemma}\label{lem:antican}
Let $\pi:Y \to X$ be a $G/B$-bundle. Then the relative anticanonical divisor $-K_{\pi}$ is a positive integral combination of the relative anticanonical divisors $-K_i$ of the elementary contractions $\rho_{i}:Y\to Y_i$. 
\end{lemma}

Given any set $J\subset D $, we denote by $\Phi^+(J)$ the subset of $\Phi^+$ consisting of positive roots that are linear combinations of the $-K_i$'s, $i\in J$. 

\begin{definition}\label{def:admiscomp}
With the same notation as above,
given a chain of subsets of $D $, $\emptyset=J_0\subsetneq J_1\subsetneq J_2\subsetneq\dots \subsetneq J_{k+1}=D$, an admissible ordering $(L_{1},\dots,L_{m})$ of $\Phi$  is said to be {\em compatible with} $J_1\subsetneq J_2\subsetneq\dots \subsetneq J_k$ if for every $r=1,\dots,k$ we have
$$
\Phi^+(J_r)=\left\{L_1,L_2,\dots,L_{\sharp(\Phi^+(J_r))}\right\}.
$$
\end{definition}

\begin{remark}\label{rem:admiscomp}
Given a chain of subsets of $D $ as above, we may always find an admissible ordering of $\Phi$ compatible with them. In fact it is enough to consider any total ordering such that the first $\sharp(\Phi^+(J_r))$ positive roots belong to $\Phi^+(J_r)$, for every $r$, and such that the order of the elements $L_{\sharp(\Phi^+(J_r))+1},\dots ,L_{\sharp(\Phi^+(J_{r+1}))}$ respects their height, for every $r$. Considering now the corresponding filtration of $T_{Y|X}$ associated with such an ordering, we may write $$\cE_{m-\sharp(\Phi^+(J_r))}=T_{Y|Y_{J_r}},\mbox{ for all }r.$$
In particular, quotienting every element of such a filtration by $T_{Y|Y_{J_r}}$ we obtain a filtration of $\rho_{J_r}^*T_{Y_{J_r}|X}$:
$$
0=\ol{\cE}_{m-\sharp(\Phi^+(J_r))}\subset \ol{\cE}_{m-1}\subset\dots\subset\ol{\cE}_{0}=\rho_{J_r}^*T_{Y_{J_r}|X},
$$
with the same quotients: $\ol{\cE}_{r}/\ol\cE_{r+1}\simeq \cE_{r}/\cE_{r+1}\cong\cO_Y(L_{m-r})$, for all $r$.
\end{remark}

\subsection{Families of rational curves }\label{ssec:ratcurves}

We finish this section by introducing some notation regarding the theory of rational curves on algebraic varieties, for which we refer the reader to \cite[Chapters~II,~IV]{kollar}. 

On a normal complex projective variety we will denote by  $\rat (X)$ the quasi-projective subvariety of the Chow variety $\Chow(X)$ of $X$ whose points correspond to irreducible and generically reduced rational curves on $X$; $\rat^n(X)$ will stand for its normalization. A {\it family of rational curves} $\cM$ on $X$ will be an irreducible  subvariety of $\rat^n(X)$; if $\cM$ is an irreducible component of $\rat^n(X)$ we will say that the family is {\it complete} (cf. \cite[II 2.11]{kollar}).

Given a family $\cM$ of rational
curves on $X$, we have the following diagram, where $p:\cU \to \cM$ is the universal family, which is known to be
a smooth $\P^1$-fibration, and $q:\cU \to X$ is the evaluation morphism:
$$
\xymatrix@=35pt{\cU\ar[r]^{q}\ar[d]_{p}&X\\\cM&}
$$

A family of rational curves is called {\it dominating} if the evaluation $q$ is dominating, and {\it unsplit} if $\cM$ is proper. We say that $X$ is {\it rationally connected} (resp. {\it rationally chain connected}\,) if through two general points of $X$ there exists a rational curve (resp. a chain of rational curves). We say that $X$ is rationally chain connected with respect to some families $\cM_1, \dots, \cM_s$ if we can take the rational curves of the chains in these families. If $X$ is smooth then $X$ is rationally connected if and only if it is rationally chain connected (see \cite[IV.3.10.3]{kollar}); we will later use that rationally connected varieties satisfy  $H^p(X, \cO_X)=0$ for $p >0$ (see  \cite[Corollary 4.18]{De}).


\section{Reducibility, decomposability, and diagonalizability}\label{sec:diagred}

A vector bundle is called  decomposable if it is a direct sum of proper vector subbundles, and this can be seen at the level of the cocycle defining it. In fact, for a vector bundle on any variety $X$ one may  consider the associated projective bundle, and its corresponding flag bundle $\pi:Y\to X$. If the vector bundle is decomposable, there exists a section of one of the corresponding Grassmannian bundles, associating to each $x \in X$ the point corresponding to one of the summands of the bundle. The existence of this section is reflected in the fact that the bundle can be defined by using block-triangular matrices, but decomposability tells us also that we have a choice of a complementary subspace at every point, so that the bundle can be defined by using block-diagonal matrix. 
Following this idea, we will introduce in this section a notion of decomposability for flag bundles.
Let us start with the following definitions:

\begin{definition}\label{def:reducible}
Let $X$ be a smooth complex projective algebraic variety, $\pi:Y\to X$ a $G/B$-bundle over $X$ defined by a cocycle $\theta\in H^1(X,G)$, and $I$ a proper subset of $D $. Then the corresponding bundle $\pi_I:Y_I\to X$ admits a section $\sigma_I:X\to Y_I$ if and only if the cocycle $\theta$ lies in the image of the natural map $H^1(X,P(I))\to H^1(X,G)$. In this case, we say that $Y$ is {\it reducible with respect to }$I$.
\end{definition}

\begin{remark}\label{rem:split}
In particular, since for every $I\subsetneq D $ the fiber product $Y_I\times_X Y_I$ admits a section (the diagonal) over $Y_I$, it follows that the pullback $\pi_I^*\theta$ belongs to the image of the map $H^1(Y_I,P(I))\to H^1(Y_I,G)$, for every $I\subsetneq D $, 
so that we may say that the pullback bundle $Y_I\times_X Y\to Y_I$ is reducible with respect to $I$, for every $I$. For $I=\emptyset$ this is the analogue of the standard Splitting Principle for vector bundles, cf. \cite[Section 3.2]{Fult}.
\end{remark}

\begin{definition}\label{def:red}
We say that a $G/B$-bundle $\pi:Y\to X$ is {\it decomposable} if there exists a proper subset 
$I\subsetneq  D $ such that:
\begin{itemize} 
\item $Y$ is reducible with respect to $I$,
\item the cocycle $\theta$ defining $\pi$, considered as an element of $H^1(X,P(I))$, 
belongs  
to the image of the natural map $H^1(X,L(I))\rightarrow H^1(X,P(I))$, where $L(I)$ is a Levi part of $P(I)$. 
\end{itemize}
Note that this map is an inclusion, since its composition with the natural map $H^1(X,P(I))\to H^1(X,L(I))$ is the identity.
\end{definition}

\begin{definition}\label{def:diag}
If the subset $I$ defining the decomposability of $\pi:Y\to X$ is empty, we say that $\pi$ is {\it diagonalizable}. The reason for this name is that the Levi parts of $B=P(\emptyset)$ are the Cartan subgroups of $G$ contained in $B$,  
hence the definition is saying that $\pi$ is defined by a cocycle in $H^1(X,(\C^*)^n)$. In particular, every vector bundle over $X$ defined by this cocycle and a given linear representation of $G$ (cf. Section \ref{ssec:bunrep}), will be a direct sum of line bundles.    
\end{definition}

\begin{remark}\label{rem:diagdis}
In the case in which the Dynkin diagram of the group $G$ is disconnected (that is, if $G$ is semisimple, but not simple), it follows that the general fiber $G/B$ is isomorphic to a product of flag varieties $G_1/B_1\times\dots\times G_k/B_k$, where every $G_i$ is a simple algebraic group. If moreover $X$ is simply connected, then the above decomposition holds globally, and $Y$ is a fiber product of flag bundles $Y_1, \dots, Y_k$ over $X$. In this case $Y$ is diagonalizable if and only if every $Y_i$ is diagonalizable.
\end{remark}

Note that any $G/B$ bundle over $\P^1$ is diagonalizable by Grothendieck's theorem (cf. \cite{Gro1}, see also Section \ref{sec:groth} below). 
On the other hand we recall that on a Fano variety of Picard number one different from $\P^1$, a rank two vector bundle is decomposable if and only if its Grothendieck projectivization admits a section.  
For varieties of this kind we will extend this result to the case of flag bundles, by showing that in this case reducibility and decomposability with respect to $\emptyset$ are equivalent (see Corollary \ref{cor:decompo1} below).  
Unfortunately, we cannot expect a similar result in the case of a general subset $I\subsetneq  D $, 
as one can see in the following example.

\begin{example}\label{ex:P(TP3)}
Let $Y$ be the complete flag over $X=\P^3$, and $\pi$ be the natural projection. As a flag bundle, it is indecomposable but, considering $I=\{2\}$, so that $Y_I\cong\P(T_{\P^3})$, the projection $\pi_I:Y_I\to X$ admits sections provided by any surjective morphism $T_{\P^3}\to\cO_{\P^3}(2)$. 
\end{example}

\subsection{Decomposability vs. reducibility}\label{decompred}

Along this section $\pi:Y\to X$ will denote a $G/B$-bundle that is reducible with respect to some $I\subsetneq  D $, and $\sigma_I:X\to Y_I$ the corresponding section. Given the associate parabolic subgroup $P:=P(I)$, we may consider the cocycle $\theta$ defining the flag bundle as an element in $H^1(X,P)$. Let us write $\theta$ as $\theta=(\theta_{ij})\in H^1(X,P)$, where $\theta_{ij}:\cV_{ij}\to P$ denote the transition functions of the bundle with respect to a trivialization on an open analytic covering $\{\cV_{i}\}$ of $X$. 

We fix now a Levi decomposition $P\cong U\rtimes L$, where $U\triangleleft P$ is the unipotent radical of $P$ and $L\subset P$ is reductive. Then the maps $\theta_{ij}$ can be written, in a unique way, as products $\theta_{ij}=\upsilon_{ij}\lambda_{ij}$, where $\upsilon_{ij}:\cV_{ij}\to U$, $\lambda_{ij}:\cV_{ij}\to L$ are holomorphic maps.

Consider the Lie algebra $\fn$ of $U$, which is the nilradical of the Lie algebra of $P$. The subgroup $L\subset P$, that we may consider as the quotient $P/U$, acts on $U$ by conjugation, inducing the adjoint action of $L$ on $\fn$, and providing the following map on cohomology:
$$
H^1(X,P)\lra H^1(X,L)\stackrel{\Ad}{\lra} H^1(X,\Aut(\fn)).
$$
The cocycle $\theta=(\theta_{ij})$ is then sent to $(\Ad_{\lambda_{ij}})\in H^1(X,\Aut(\fn))$, which defines a vector bundle on $X$, that we denote by $\mho_{I}$.

Note that, since $U$ is unipotent, the exponential map from its Lie algebra $\fn$ is bijective, and we have an inverse $\log:U\to \fn$, that we may use to define, for every pair of indices $i,j$, a holomorphic map $\xi_{ij}:=\log(\upsilon_{ij}):\cV_{ij}\to\fn$. 

Now, taking the unipotent part on the cocycle condition on $\theta$:
$$
\theta_{ij}\theta_{jk}=\theta_{ik},\mbox{ for all }i,j,k,
$$
we get:
$$
\upsilon_{ij}\lambda_{ij}\upsilon_{jk}\lambda_{ji}=\upsilon_{ik},\mbox{ for all }i,j,k,
$$
and so:
$$
\xi_{ij}+\Ad_{\lambda_{ij}}(\xi_{jk})=\xi_{ik},\mbox{ for all }i,j,k,
$$
which may be rephrased as follows:
\begin{lemma}
The collection $(\xi_{ij})$ defines a cocycle in $H^1(X,\mho_I)$.
\end{lemma}

\begin{proof}
In fact, we may consider each $\xi_{rs}:\cV_{rs}\to\fn$ as a section of $\mho_I(\cV_{rs})$, via the inclusion $\cV_{rs}\subset \cV_{r}$, and the corresponding trivialization of $\mho_I$ over $\cV_r$. In the above formula, $\xi_{ij}$ and $\xi_{ik}$ correspond to sections of $\mho_I$ on the open set $\cV_{ijk}$, expressed as maps $\cV_{ijk}\subset \cV_{i}\to \fn$, via the trivialization of $\mho_I$ over $\cV_i$. In turn, the map $\xi_{jk}:\cV_{ijk}\to\fn$ provides a section of $\mho_I$ via $\cV_{ijk}\subset \cV_{j}$, and we must use the transition $\Ad_{\lambda_{ij}}$ to think of it in terms of the same trivialization as $\xi_{ij}$ and $\xi_{ik}$. 
\end{proof}

We may then state the following:

\begin{proposition}\label{prop:decompo1}
Let $\pi:Y\to X$ be a flag bundle, reducible with respect to some subset $I\subsetneq D $. Then $\pi$ is decomposable with respect to $I$ if and only if the cocycle $\xi\in H^1(X,\mho_I)$  is zero.
\end{proposition}

\begin{proof}
If $(\xi_{ij})\in H^1(X,\mho_I)$ is zero, then there exist, for every index $i$, a map $\phi_i:\cV_i\to \fn$ such that, for every $i,j$:
$$
\xi_{ij}=\phi_i-\Ad_{\lambda_{ij}}(\phi_j).
$$
By means of the exponential map, we may write this as:
$$
\upsilon_{ij}=\exp(\phi_i)\lambda_{ij}\exp(-\phi_j)\lambda_{ji},
$$
that is 
$$
\theta_{ij}=\exp(\phi_i)\lambda_{ij}\exp(-\phi_j).
$$
This implies that the cohomology class of $(\theta_{ij})$ lies in $H^1(X,L)$. The converse is obvious.
\end{proof}

\begin{remark} This result can be seen as a generalization of the case of vector bundles, too. In fact, the obstruction for a vector bundle $\cE$ over $X$ given as an extension 
$$
\shse{\cE'}{\cE}{\cE''}
$$
to be a direct sum of $\cE'$ and $\cE''$ lies in $H^1(X,\cE'\otimes\cE'')$, which is precisely the $H^1$ of the restriction of the relative cotangent bundle of the associated Grassmannian bundle to its section over $X$ provided by the quotient $\cE\to\cE''$; that relative cotangent bundle coincides with the bundle $\mho_I$ defined above, in this case. 
\end{remark}
 
\begin{corollary}\label{cor:decompo1}
Let $X$ be a smooth variety, and $\pi:Y\to X$ be a $G/B$-bundle admitting a section $\sigma:X\to Y$. If $H^1(X,\sigma^*R)=0$ for every $R\in \Phi^-$, then $\pi$ is diagonalizable. In particular, if $X$ is a Fano manifold of Picard number one, $\pi$ is diagonalizable if and only if it admits a section.
\end{corollary}

\begin{proof}
Levi parts $L\subset B$ are Cartan subgroups of $G$, hence the action of $L$ on $\fn$ decomposes as a direct sum of subrepresentations of dimension one, each of them corresponding to a negative root of the group, and to a line bundle $\cO_X(\sigma^*R)$. We then have the equality:
$$
\mho_{\emptyset}\cong\bigoplus_{R\in\Phi^-}\cO_X(\sigma^*R),
$$ 
from which we get the first assertion of the statement. 

For the second part, assume that $X$ is a Fano manifold of Picard number one. If $\pi$ admits a section then, in the case $X\cong \P^1$ it is diagonalizable by Grothendieck's theorem while, if $\dim(X)\geq 2$ the result follows by the first part of this statement and Kodaira vanishing. Conversely, if $\pi$ is diagonalizable then its defining cocycle is in the image of the natural map $H^1(X,H)\to H^1(X,G)$, for some Cartan subgroup $H$. Then the cocycle is also in the image of the natural map $H^1(X,B)\to H^1(X,G)$, which in turn implies that the flag bundle has a section (see Section \ref{ssec:stconst}).
\end{proof}

Finally, we may also consider the group $G'=L/Z(L)$, which is a semisimple linear algebraic group. The image of $\theta$ into  
$H^1(X,G')$ defines a flag bundle on $X$, that we denote here by $\pi':Y'\to X$, fitting in the following diagram:

$$
\xymatrix@=35pt{Y\ar[r]_{\rho_I}\ar@/^{4mm}/[rr]^\pi&Y_I\ar[r]&X\\
Y'\ar[u]\ar[r]_{\pi'}&X\ar@{=}[ru]\ar[u]^{\sigma_I}&}
$$
Since a section of $\pi'$ gives a section of $\pi$, 
Corollary \ref{cor:decompo1}, provides the following:

\begin{proposition}\label{prop:decompo2}
Let $X$ be a Fano manifold of Picard number one, and $\pi:Y\to X$ be a flag bundle, reducible with respect to some subset $I\subsetneq D $. Then the $G'/B'$-bundle $\pi'$ defined above is diagonalizable if and only if $\pi$ is diagonalizable. 
\end{proposition}

\section{Flag bundles on the projective line}\label{sec:groth}

Let us discuss here some basic facts on $G/B$-bundles over the projective line, which are diagonalizable by Grothendieck's theorem (cf. \cite{Gro1}). As noted precedently, we assume that $G=\Aut^\circ(G/B)$, and we denote by $H\subset G$ a Cartan subgroup of $G$. Grothendieck's theorem tells us that the natural map $H^1(\P^1,H)\to H^1(\P^1,G)$ is surjective, so the cocycle $\theta\in H^1(\P^1,G)$ defining a given $G/B$-bundle $\pi:Y\to \P^1$ comes from an element in $H^1(\P^1,H)$ (that we also denote by $\theta$). Moreover, $H^1(\P^1,H)$ is naturally isomorphic to the {\em lattice of co-characters} of $H$, $$\Na(H):=\Hom_\Z(\Mo(H),\Z),$$ and Grothendieck noted that the fibers of the map from $\Na(H)$ onto $H^1(\P^1,G)$ are precisely the orbits of the induced action of the Weyl group $W$ on $\Na(H)$. 

%

One may interpret this result geometrically on the associated $G/B$-bundles (see \cite[Section~3.3]{OSW}). If $\pi:Y\to \P^1$ is the $G/B$-bundle associated with a cocycle $\theta\in H^1(X,H)\cong\Na(H)$, we may construct a 
section as follows: 
any choice of a Borel subgroup $B\supset H$ (there are as many of them as elements of the Weyl group of $G$), together with the cocycle $\theta$, determines a section $\Gamma_0$ of the $G/B$-bundle $\pi:Y\to \P^1$. A choice of $B$ corresponds to the choice of a set of positive simple roots $\Delta=\{\alpha_i,\,\,i=1,\dots,n\}$. Choosing $\Delta$ so that $d_i:=-\theta(\alpha_i)\geq 0$, for every $i$, the resulting section is a minimal section of $\pi$, called a {\em minimal fundamental section}. The section $\Gamma_0$ is minimal, in the sense that its deformations with a point fixed are trivial, and the integers $d_i$ are equal to the intersection numbers $K_i\cdot\Gamma_0$. The choice of $B$ may not be unique, but one may show that the numerical class of minimal sections is unique, in any case (\cite[3.20, 3.21, 3.22]{OSW}).

Then the flag bundle $\pi:Y\to \P^1$ is completely determined by the Dynkin diagram $\cD$ of $G$ (determining the type of flag manifold $G/B$ appearing as fibers of the bundle), together with the $n$-tuple $\delta=(d_1,\dots,d_n)$, that we call {\em tag of the flag bundle} $\pi$. The ordering of the tag depends on the order of the simple roots of $\Delta$, which is in one to one correspondence with the nodes of $\cD$, so it makes sense to represent the flag bundle by a {\em tagged Dynkin diagram}, that is, the Dynkin diagram $\cD$ decorated with the integer $d_i$ at the node corresponding to $\alpha_i$, for each $i$.

\begin{remark}\label{rem:tagsplit}
In particular, in the case in which $G/B$ is the complete flag manifold of a projective space, the tag of a $G/B$-bundle over $\P^1$ can be easily computed from the {\it splitting type} of a vector bundle $\cF$ defining the corresponding projective bundle. The bundle $\cF$ is isomorphic to a direct sum of line bundles, and the splitting type of $\cF$ is defined as the $r$-tuple $(a_1,\dots, a_r)$, $a_1\leq\dots \leq a_r$, of the degrees of these line bundles. Then one may easily check that the tagged Dynkin diagram of the $G/B$-bundle is:
$$
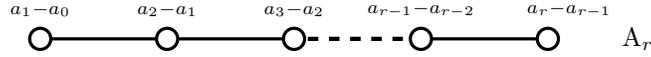
\ifx\du\undefined
  \newlength{\du}
\fi
\setlength{\du}{4\unitlength}
\begin{tikzpicture}
\pgftransformxscale{1.000000}
\pgftransformyscale{1.000000}

\definecolor{dialinecolor}{rgb}{0.000000, 0.000000, 0.000000} 
\pgfsetstrokecolor{dialinecolor}
\definecolor{dialinecolor}{rgb}{0.000000, 0.000000, 0.000000} 
\pgfsetfillcolor{dialinecolor}


\pgfsetlinewidth{0.300000\du}
\pgfsetdash{}{0pt}
\pgfsetdash{}{0pt}

\pgfpathellipse{\pgfpoint{-6\du}{0\du}}{\pgfpoint{1\du}{0\du}}{\pgfpoint{0\du}{1\du}}
\pgfusepath{stroke}
\node at (-6\du,0\du){};

\pgfpathellipse{\pgfpoint{6\du}{0\du}}{\pgfpoint{1\du}{0\du}}{\pgfpoint{0\du}{1\du}}
\pgfusepath{stroke}
\node at (6\du,0\du){};

\pgfpathellipse{\pgfpoint{18\du}{0\du}}{\pgfpoint{1\du}{0\du}}{\pgfpoint{0\du}{1\du}}
\pgfusepath{stroke}
\node at (18\du,0\du){};

\pgfpathellipse{\pgfpoint{30\du}{0\du}}{\pgfpoint{1\du}{0\du}}{\pgfpoint{0\du}{1\du}}
\pgfusepath{stroke}
\node at (30\du,0\du){};

\pgfpathellipse{\pgfpoint{42\du}{0\du}}{\pgfpoint{1\du}{0\du}}{\pgfpoint{0\du}{1\du}}
\pgfusepath{stroke}
\node at (42\du,0\du){};


\pgfsetlinewidth{0.300000\du}
\pgfsetdash{}{0pt}
\pgfsetdash{}{0pt}
\pgfsetbuttcap

{\draw (-5\du,0\du)--(5\du,0\du);}
{\draw (7\du,0\du)--(17\du,0\du);}
{\draw (31\du,0\du)--(41\du,0\du);}



\pgfsetlinewidth{0.400000\du}
\pgfsetdash{{1.000000\du}{1.000000\du}}{0\du}
\pgfsetdash{{1.000000\du}{1.000000\du}}{0\du}
\pgfsetbuttcap
{\draw (19.3\du,0\du)--(29\du,0\du);}

\node[anchor=west] at (48\du,0\du){${\rm A}_r$};

\node[anchor=south] at (-6\du,1.1\du){$\scriptstyle a_1-a_0$};

\node[anchor=south] at (6\du,1.1\du){$\scriptstyle a_2-a_1$};

\node[anchor=south] at (18\du,1.1\du){$\scriptstyle a_3-a_2$};

\node[anchor=south] at (30\du,0.95\du){$\scriptstyle a_{r-1}-a_{r-2}$};

\node[anchor=south] at (44\du,0.95\du){$\scriptstyle a_r-a_{r-1}$};


\end{tikzpicture} 
$$
\end{remark}

\subsection{Splitting type of vector bundles associated with representations}\label{ssec:splittype}

Let us now consider here a $G$-principal bundle over $\P^1$, determined by an element $\theta\in \Na(H)$, and a linear representation $V$ of $G$ of dimension $r$, whose associated vector bundle, denoted by $\cF$, has splitting type $(a_1,\dots, a_r)$. 
The relation between this splitting type and the data defining the representation $V$ --its weights $w_1,\dots,w_r\in \Mo(H)$ (possibly repeated)
-- is the following: if the $G$-principal bundle over $\P^1$ is given by the lattice point $\theta\in\Na(H)$, then the entries of the splitting type are precisely the products $\theta(w_i)$, $i=1,\dots,r$.

The situation is similar if we consider a projective representation $\P(V)$ of $G$: it defines a projective bundle which, being the base $\P^1$, can be written as $\P(\cF)$ for some vector bundle $\cF$, and we are interested here in computing its splitting type, up to a twist. 

Note that the twists of the vector bundle $\cF$ are necessarily coming from a linear representation $V$ of $G$; in fact, the action of $G$ over $\P(V)$ does not necessarily lift to an action on $V$, but to the action of another semisimple group $G'$ 
(in particular, $V$ is a representation of the Lie algebra $\fg$)
with the same Lie algebra, related to $G$ via an isogeny $f:G'\to G$. In particular, we may consider the weights of this linear representation of $G'$, which lie in $\Mo(H')$, but not necessarily in $\Mo(H)$. If $\theta\in\Na(H')\subset\Na(H)$ (that is if our $G$-principal bundle can be reduced to $G'$), then the splitting type of $\cF$ can be computed as above. If that is not the case we may still compute the splitting type as follows. 

Let $(a_1,\dots, a_r)$  be the splitting type of $\cF$, 
and let $w_1,\dots, w_r\in P$ be the weights of the representation $V$ of $G'$.  
Since $\Mo(H)$ has finite index in $\Mo(H')$ it follows that there exists a positive integer $m$ such that the symmetric power $S^mV$ is a representation of $G$. 
This representation has as weights the sums $\sum_{i=1}^mw_{r_i}$, $r_i\in\{1,\dots,r\}$, and it provides a vector bundle $\cG$, which is isomorphic to a twist of the symmetric power $S^m\cF$. 
The elements of the splitting type of $\cG$, that we denote by  $(b_1,\dots, b_t)$, are of the form $\sum_{i=1}^m\theta(w_{r_i})$, $r_i\in\{1,\dots,r\}$. On the other hand, it follows that there exists a constant $C\in \Z$ such that every $b_i$ is of the form $C+\sum_{i=1}^ma_{r_i}$, $r_i\in\{1,\dots,r\}$.

Let us observe now that the map associating to every set of real numbers $\{a_1,\dots,a_r\}$ the set $\{\sum_{i=1}^ma_{r_i}|\,r_i=1,\dots,r\}$ of all the possible sums of $m$ (possibly repeated) elements of the set, is injective. 
It then follows that, up to reordering of the indices, we may write $\theta(w_i)=a_i+C/m$. As a consequence the tag of $\P(\cF)$ is given by the successive difference of the elements $\theta(w_i)$. Summing up:

\begin{lemma}
With the same notation as above, let $\P(\cF)$ be a projective bundle over $\P^1$ determined by a $G$-principal bundle defined by a co-character $\theta\in \Na(H)$, and by a projective representation of $G$, given by weights $w_1,\dots,w_r$. Assume that the indexing of these weights is chosen so that $\theta(w_1)\leq\dots\leq \theta(w_r)$. Then the tag of the flag bundle associated to $\P(\cF)$ is precisely 
$$\big(\theta(w_2)-\theta(w_1),\dots,\theta(w_r)-\theta(w_{r-1})\big).$$ 
\end{lemma}

\subsection{The Mori cone of a flag bundle over $\P^1$}\label{ssec:minsecgroth}

In general, given a smooth morphism $\pi:Z\to \P^1$, a section $\sigma:\P^1\to Z$ of $\pi$ is called {\it minimal} if it cannot be deformed with a point fixed. In particular, a section $\sigma$ is minimal if $H^0(\P^1,\sigma^*T_{Z|\P^1})=0$. 

Minimal sections of a projective bundle $\P(\cF)=\P(\bigoplus_i\cO(a_i))$ over $\P^1$ are classically known to be in 
one-to-one correspondence with the quotients $\bigoplus_i\cO(a_i)\to\cO(a_{-})$, where $a_{-}=\min\{a_i\}$. 
In particular, the locus of these minimal sections is of the form $\P^1\times\P(V_{-})\subset\P(\cF)$, for a certain vector space $V_{-}$, and  minimal sections correspond to fibers of the natural projection $\P^1\times\P(V_{-})\to \P(V_{-})$.  

\begin{proposition}\label{prop:moricone}
Let $\pi:Y \to \P^1$ be a flag bundle. Denote by $R_i:=\R_+[\Gamma_i]$, $i=1, \dots n$,  the extremal ray in $\NE(Y)$ corresponding to the elementary contraction $\rho_i:Y \to Y_i$, and by $\Gamma_0$ a minimal fundamental section of $\pi$. Then the Mori cone of $Y$ is the simplicial cone generated by the rays $R_i$, $i=1,\dots, n$, and $\R_+[\Gamma_0]$.
\end{proposition}

\begin{proof}

Since we know that the relative Mori cone of $Y$ over $\P^1$ is a simplicial cone whose extremal rays are the $R_i$'s (see Section \ref{sec:prelim}), 
it is enough to show that, for any subset $I\subset D$ (including the case $I=\emptyset$) there exists a morphism $f_I:Y\to Y'_I$ contracting precisely the curves whose classes lie in the vector subspace of $\Nu(Y)$ generated by $[\Gamma_0]$ and the rays $R_i$, $i\in I$. 

In order to do so, we consider a weight $w$ belonging to the relative interior of $$\ol{\cW}\cap\langle\{\lambda_j,\,\,j\notin I\}\rangle\subset \langle\{\lambda_j,\,\,j\notin I\}\rangle,$$ where $\ol{\cW}$ denotes the closure of the fundamental Weyl chamber $\cW$. In this way, being $V$ the irreducible representation of $G$ with highest weight $w$, the variety $G/P(I)$ is the unique closed $G$-orbit in $\P(V)$. Let us denote by 
$w_1=w,\dots, w_r$ the weights of $V$ as a representation of $\fg$, and $V=\bigoplus_{i=1}^rV_i$ the corresponding eigenspace decomposition of $V$ (so that $V_i$ is the eigenspace associated to the weight $w_i$). Note that the $w_i$'s are contained in the polytope $P(w)$ generated by the orbit $W.w$ of $w$ by the action of the Weyl group $W$ in $\Mo(H)\otimes_\Z\R$. Moreover, $P(w)$ is contained in the cone $w+\R_{\geq 0}(-\alpha_1,\dots,-\alpha_n)$; since $d_i=\theta(-\alpha_i)\geq 0$ and $w\in\ol{\cW}$ is a non negative combination of the roots $\alpha_i$'s, one can easily show that $a_{-}:=\min(\theta(w_i))=\theta(w)$. Set now
$$
V_{-}:=\bigoplus_{\theta(w_i)=a_{-}} V_i,
$$  
and denote by $\cE_H$ the $H$-principal bundle associated to the cocycle $\theta\in H^1(\P^1,H)$ defining $\pi$. We have two embeddings $$Y_I\hookrightarrow\P(\cF):=\cE_H\times^H\P(V),\quad\cE_H\times^H\P(V_{-})\hookrightarrow\P(\cF)
,$$ 
and $\cE_H\times^H\P(V_{-})\cong\P^1\times\P(V_{-})$ is the locus of minimal sections of $\P(\cF)$.

Note that, being $w$ the highest weight of the representation $V$, $\dim(V_w)=1$ and so $\cE_H\times^H\P(V_{w})\subset\P^1\times\P(V_{-})$ is a minimal section of $\P(\cF)$. On the other hand, following \cite[p. 388]{FuHa}, the unique closed orbit of the action of $G$ on $\P(V)$, isomorphic to $G/P(I)$, is the orbit of the point $P(V_w)\in \P(V)$, corresponding to the natural projection $V\to V_{w}$. The fact that $w\in\ol{\cW}$ implies that the corresponding eigenspace in $V$ is fixed by 
the Borel subgroup $B$, and we may conclude that $B$ is contained in the isotropy subgroup of $p$, that is $P(I)\subset G$. Then we have $H\subset B\subset P(I)$, and so $\cE_H\times^H\P(V_{w})$ is the image of a minimal fundamental section $\Gamma_0$ of $Y$ into $Y_I$. 

By taking an appropriate twist, we may assume that the vector bundle $\cF$ is nef, but not ample, and so the
natural map $c:\P(\cF)\to\P(H^0(\P^1,\cF))$ defined by the evaluation of global sections contracts the curves whose classes lie in the ray generated by the class of the section $\cE_H\times^H\P(V_{w})$.  We may then consider the composition $c\circ \rho_I:Y\to \P(H^0(\P^1,\cF))$, which is a contraction satisfying the required properties.
\end{proof}

From this statement, by means of the contractions $\rho_I$ of the flag bundle defined in Section \ref{sec:prelim}, 
we get a similar result for any rational homogeneous bundle over $\P^1$:

\begin{corollary}  $\pi:Y_I \to \P^1$ be a rational homogeneous bundle with fibers $G/P(I)$. Then, with the notation of Proposition \ref{prop:moricone}, the cone $\NE(Y_I)$ is generated by the images  via $\rho_{I*}$ of the rays $R_k, \,\,k \in D$, and  $\R_+[\Gamma_0]$.
\end{corollary}

We end this section with the description of the locus of minimal sections of a flag bundle, which turns out to be a trivial subflag. 
Let $\pi: Y \to \P^1$ be a flag bundle, with tag $\delta=(d_1, \dots, d_n)$, and set
$$
I_0:=\left\{i\in  D |\,\,d_i=0\right\}.
$$
The  Dynkin subdiagram of $\cD$ supported on $I_0$ will be denoted by $\cD_{I_0}$ and $P(I_0)\subset G$ will stand for the corresponding parabolic subgroup (so that the fibers $F_{I_0}$ of the submersion $\rho_{I_0}:Y\to Y_{I_0}$ are flag manifolds associated with a semisimple subgroup of $G$ determined by the Dynkin subdiagram $\cD_{I_0}$).  Then

\begin{proposition}\label{prop:trivialsub}
The locus of minimal sections of $\pi:Y \to \P^1$ is a trivial subflag bundle $F_{I_0}\times \P^1$.
In particular the minimal sections of $Y$ are algebraically equivalent. 
\end{proposition}

\begin{proof}
As in the proof of Proposition \ref{prop:moricone} consider a highest weight $w$ belonging to the relative interior of $$\ol{\cW}\cap\langle\{\lambda_j,\,\,j\notin I_0\}\rangle\subset \langle\{\lambda_j,\,\,j\notin I_0\}\rangle,$$ 
denote by  $w_1=w,\dots, w_r$ be the weights of $V$, and by $a_-$ the minimum of $\theta(w_i)$ (which equals $\theta(w))$.

Denoting by  $W_{I_0}$ the subgroup of the Weyl group generated by the set of reflections $\{r_{i}|\,\, i\in I_0\}$,
  we have that $(\theta(w_i))=a_-$ if and only if $w_i \in W_{I_0}.w$ and this happens if and only if $w_i=w$,  since  $w \in \langle\{\lambda_j,\,\,j\notin I_0\}\rangle$.
  
In particular the projective bundle $\P(\cF)$ (defined by the projective representation of $\fg$ with highest weight $w$) has only one minimal section $\Gamma$, and therefore the same happens for $Y_{I_0}$. Over this minimal section the flag bundle $\rho_{I_0}$ is trivial, hence, if $\Gamma_0$ is a minimal section of $Y$ mapping to $\Gamma$ we have  $\rho_{I_0}^{-1}(\Gamma)= \Gamma_0 \times F_{I_0}$.
\end{proof}


\section{Uniform flag bundles}\label{sec:unif}

Uniformity of flag bundles is an extension of a classical concept within the theory of vector bundles (cf. \cite[\S 3]{OSS}), that applies to a triple $(X,\cM,\cE)$, where $X$ is an algebraic variety, $\cM$ is a dominating family of rational curves on $X$, and $\cE$ is a vector bundle on $X$. Then $\cE$ is said to be uniform with respect to  $\cM$ if the (isomorphism class of the) pullback of $\cE$ via the normalization of one of the curves of the family does not depend on the  chosen curve.

Let us now consider a $G/B$-bundle $\pi:Y\to X$ on a smooth complex projective variety $X$, and a dominating  family of rational curves $\cM$ on $X$, with universal family  $p:\cU \to \cM$, and evaluation morphism $q:\cU\to X$. We may consider the pullback $q^*Y=Y\times_X\cU$, which is a $G/B$-bundle over $\cU$, whose natural morphism onto $\cU$ will be denoted by $\pi$, by abuse of notation. Following Section \ref{sec:groth}, for every rational curve $\Gamma=p^{-1}(z)\subset\cU$ the  pullback of the $G/B$-bundle $q^*Y$  to $\Gamma$ is determined by its tag $\delta_\Gamma(Y):=(d_1,\dots,d_n)\in\Z_{\geq 0}^n$ on the Dynkin diagram of the group $G$, and one may pose the following definition:



\begin{definition}\label{def:redu}
Given a smooth complex projective variety $X$, a dominating family of rational curves $\cM$ on $X$, and a flag bundle $\pi:Y\to X$, we say that $Y$ is {\it uniform with respect to $\cM$} if the tag $\delta_\Gamma(Y)$ is independent of the choice of the curve $\Gamma\in\cM$. In this case, the tag will be denoted  by $\delta(Y)$, or simply by $\delta$. 
\end{definition}

\begin{remark} Note that, if $\pi:Y\to X$ is a uniform flag bundle defined by a cocycle $\theta\in H^1(X,G)$,
then every vector bundle determined by $\theta$ and a given linear representation of $G$ will be uniform in the classical sense.
\end{remark}

\begin{example}
Besides complete flag bundles defined by uniform projective bundles, the most obvious examples of uniform flag bundles are the products $X\times G/B$, that we call {\it trivial} flag bundles. Moreover, given a semisimple group $G$, and a maximal parabolic subgroup $P\subset G$ corresponding to a simple root that is not exposed short (\cite[Definition 2.10]{LM}), then the map $G/B\to G/P$ is  a $G'/B'$-bundle, where $G'$ is a Levi part of $P$, and $B'$ is a Borel subgroup of $G'$, that is uniform with respect to the complete family of lines in $G/P$. 
\end{example}

\subsection{Families of minimal sections}

In this section we consider a  flag bundle $\pi:Y \to X$, uniform with respect to an unsplit dominating family of rational curves $\cM$, and we prove that the set of minimal sections of $Y$ over curves of $\cM$ defines an unsplit family of rational curves.

Consider ${\cM'}$ to be a component of $\rat^n(Y)$ containing a general minimal section over a general curve of $\cM$; by construction the closure of the image of ${\cM'}$ via the natural morphism ${\pi'}:{\cM'} \to \rat^n(X)$ contains $\cM$; denote by $\ol{\cM}$ the inverse image of $\cM$ via ${\pi'}$. 

\begin{proposition}\label{prop:minsec}
The subvariety $\ol{\cM} \subset \rat^n(X)$ is proper and irreducible, and parametrizes all the minimal sections over curves in $\cM$.
\end{proposition}

\begin{proof}
Assume, by contradiction, that  curves in $\ol{\cM}$ degenerate to a non integral 1-cycle, i.e., that there exists a complex neighbourhood $0 \in C \subset \C$ and a map  $f:C \to \Chow(Y)$ such that, for $t \not =0$ the point $f(t)$ belongs to (the image of) $\ol{\cM}$, while the point $f(0)$ parametrizes a non integral 1-cycle. 

Since $\cM$ is proper, the image of $C$ in $\Chow(X)$ is contained in $\cM$, and the evaluation morphism $q:C\times \P^1\to X$  induces a rational map $q':C\times \P^1\dashrightarrow q^*Y$, where $q^*Y$ denotes the fiber product of the bundle $\pi:Y\to X$ and the map $q$. 

We may then consider the resolution of indeterminacies of $q'$ performed by a sequence of blowups over points in $\Gamma:=\{0\}\times\P^1$. This construction provides an effective $1$-cycle $Z$ in $q^*Y$, dominating $\Gamma$, algebraically equivalent to $q'(t,\P^1)$, $t\neq 0$, which is a minimal section of $q^*Y$ over the normalization of a curve in $\cM$. By the uniformity of $Y$, the cycle $Z$ is numerically equivalent to $\Gamma_0$, where $\Gamma_0$ denotes a minimal  section of $q^*Y$ over $\Gamma$.

Writing $Z=\sum_{i=1}^m a_iZ_i$, being $Z_i$ the classes of the irreducible components of $Z$, and $a_i\in \Z_+$, the description of the Mori cone given in Proposition \ref{prop:moricone} implies that $Z_i\in\R_+[\Gamma_0]$ for all $i$. On the other hand, $\Gamma_0$ is a minimal degree curve in the ray
$\R_+[\Gamma_0]$, hence $Z$ must consist of a unique irreducible component of multiplicity one. This completes the proof of the properness of $\ol{\cM}$.

Let $\ol{\cM}_1$ be an irreducible component of $\ol{\cM}$ containing a  general minimal section over a general element of $\cM$. By Proposition \ref{prop:trivialsub} the fiber of $\ol{\cM}_1 \to \cM$ over this point is isomorphic to the subflag $F \subset G/B$. By semicontinuity the fiber of $\ol{\cM}_1 \to \cM$ over any element $\ell$ has dimension at least $\dim F$; on the other hand, again by Proposition \ref{prop:trivialsub}, the fiber over $\ell$ of $\ol{\cM} \to \cM$ is contained in $F$. It follows that $\ol{\cM}_1 = \ol{\cM}$, and that $\ol{\cM}$ parametrizes all the minimal sections over curves in $\cM$.
\end{proof}

\subsection{Characterization of trivial flag bundles}\label{ssec:trivial}

We will now consider the simplest case in which the uniformity of the flag bundle allows us to classify it, i.e., the case in which $\delta(Y)=(0,\dots,0)$.

Throughout this section $ \pi:Y\to X$ will denote a flag bundle over a smooth complex projective manifold $X$and we will consider $s$  families of rational curves $$\cM_i\stackrel{p_i}{\longleftarrow} \cU_i\stackrel{q_i}{\longrightarrow} X.$$ 
We may now prove the characterization of trivial flag bundles stated in the introduction:
\begin{theorem}\label{thm:trivial1}
Let $X$ be a manifold which is rationally chain connected with respect to  $\cM_1, \dots,\cM_s$, unsplit families of rational curves, and $\pi:Y\to X$ a $G/B$ bundle over $X$. 
Assume that for every rational curve $\Gamma_i:=p_i^{-1}(z)$ we have $\delta_{\Gamma_i}(Y)=(0, \dots, 0)$.
Then $ Y\cong X\times G/B$ is trivial as a $G/B$-bundle over $X$.
\end{theorem}

\begin{proof}
For every $i=1,\dots, s$, we consider the family $\ol{\cM}_i$ of minimal sections of $Y$ over curves of $\cM_i$,
constructed in Proposition \ref{prop:minsec}.

These families  define a {\it rational quotient} of $Y$, i.e., there exists a proper morphism $\tau: Y^0\to Z^0$, defined on an open set $ Y^0\subset Y$, onto a normal variety $Z^0$, whose fibers are  equivalence classes in $Y^0$ of the relation defined by connectedness with respect to the families $\overline{\cM}_i$ (see \cite[IV.4.16]{kollar} for details). 

A general fiber $X'$ of $\tau$ is a smooth projective variety, rationally connected by the curves of the (unsplit) families $\overline{\cM}_i$. 
This implies that the numerical class of every curve contained in $X'$ is a linear combination of the numerical classes of the curves parametrized by the families $\overline{\cM}_i$'s. In particular $-K_j$ is numerically trivial on $X'$ for every $j \in  D $, hence trivial, since being $X'$ rationally connected, $H^1(X',\cO_{X'})=0$, and so the natural map $\Pic(X')\to H^2(X',\Z)$ is injective. Therefore
$-K_\pi$, which is an integral combination of the $-K_j$'s (see Lemma \ref{lem:antican}) is trivial on $X'$.

We claim now that the restriction of $\pi$ to $X'$ is necessarily finite onto $X$. The finiteness follows from the fact that $X'$ cannot contain a curve contracted by $\pi$, since $-K_\pi$ is $\pi$-ample, while
the surjectivity follows by the interpretation of each $\overline{\cM}_i$ as the family of minimal sections over curves of $\cM_i$, the triviality of $ Y$ on these curves, and the 
rational chain connectedness of $X$ with respect to them. 


Now, adjunction tells us that $$K_{X'}=(K_{Y})_{|X'}=\big(K_{\pi}+ \pi^*K_X\big)_{|X'}=( \pi^*K_X)_{|X'},$$
so $\pi_{|X'}$ is an \'etale cover of $X$, contradicting that $X$ is rationally chain connected, and hence simply connected, unless $X'$ is a section of $\pi$.

By Corollary \ref{cor:decompo1} we may conclude that $ \pi: Y\to X$ is diagonalizable, i.e., $\pi$ is defined by a cocycle in $H^1(X,(\C^*)^n) \simeq \Pic(X)^n$; let $L_1,\dots, L_n\in\Pic(X)$ be the line bundles in $X$ determined by this cocycle. Since the restriction of $Y$ to any rational curve of the families $\cM_i$ is trivial, it follows that $L_1,\dots,L_n$ are trivial on each one of this curves. But $X$ is rationally chain connected with respect to the families $\cM_i$, therefore the line bundles $L_j$ are numerically trivial. Finally,  since $X$ is rationally connected, arguing as in the case of $X'$ above, 
we conclude that the line bundles $L_j$ are trivial, which is equivalent to say that the cocycle determining the bundle is trivial. 
\end{proof}

As a consequence of  Theorem \ref{thm:trivial1}, taking in account that a rational homogeneous bundle is trivial if and only if its associated flag bundle is trivial we get:

\begin{corollary}\label{cor:trivial2}
Let $X$ and $\cM_1, \dots,\cM_s$ be as in Theorem \ref{thm:trivial1}, and let $\pi:E\to X$ be an $F$-bundle over $X$, with $F$ rational homogeneous, satisfying that for the normalization $f_i:\P^1\to X$ of any curve of the family $\cM_i$ and all $i=1,\dots,s$, the fiber product $\P^1\times_X Y$ is trivial as an $F$-bundle over $\P^1$. Then $E$ is trivial as an $F$-bundle over $X$.
\end{corollary}

\section{Diagonalizability criteria for uniform flag bundles}\label{sec:diagunif}

Along this section $X$ will denote a Fano manifold of Picard number one and $\pi:Y\to X$ a flag bundle, uniform with respect to an unsplit dominating  family $\cM$ of rational curves, with tag $\delta=(d_1,\dots,d_n)$. 

Considering the family $\ol{\cM}$ of minimal sections of this bundle over the curves of the family $\cM$  leads to a concept of {\it uniform reducibility}, that we will discuss in Section \ref{ssec:redu}. 
Then, in Section \ref{ssec:GMcrit}, we will state some reducibility criteria for uniform flag bundles. In particular, we will show a flag bundle counterpart of the classic Grauert--M\"ulich theorem, together with some diagonalizability criteria for uniform bundles with special tagging.

\subsection{Uniform reducibility of uniform flag bundles}\label{ssec:redu}

As in Section \ref{ssec:minsecgroth} we set
$I_0:=\left\{i\in  D |\,\,d_i=0\right\}$; by the uniformity assumption this is the same on every curve of $\cM$.
In view of Theorem \ref{thm:trivial1}, we will always assume $I_0\subsetneq  D $. 

%

Let us then consider the family of minimal sections of $Y$ over curves of the family $\cM$, denoted by $\ol{p}:\ol{\cU}\to\ol{\cM}$.
There is a commutative diagram:
$$
\xymatrix@=25pt{\ol{\cU}\ar[r]^{\ol{p}}\ar[d]&\ol{\cM}\ar[d]\\\cU\ar[r]^p&\cM}
$$
whose vertical arrows are smooth morphisms whose fibers, by Proposition \ref{prop:trivialsub}, are isomorphic to the subflag $F_{I_0}$.
We may consider the composition of the evaluation $\ol{q}:\ol{\cU}\to Y$ with $\rho_{I_0}$, that is constant on the fibers of $\ol{\cU}\to\cU$, and so we obtain a map $s_0:\cU\to Y_{I_0}$, fitting in the following commutative diagram:
$$
\xymatrix@C=45pt@R=25pt{\ol{\cM}\ar[dd]&\ol{\cU}\ar[l]_{\ol{p}}\ar[r]^{\ol{q}}\ar[rd]\ar[dd]&Y
\ar[d]^{\rho_{I_0}}\\&&Y_{I_0}\ar[d]^{\pi_{I_0}}\\\cM&\cU\ar[ur]^{s_0}\ar[l]_{p}\ar[r]^{q}&X}
$$

\begin{definition}\label{def:decomp2}
With the same notation as above, given any set $I\subsetneq  D $  of nodes of $\cD$ containing $I_0$, we denote by $\rho_{I_0,I}:Y_{I_0}\to Y_I$ the natural projection. We say that $\pi:Y\to X$ is {\it uniformly reducible with respect to $\cM$ and $I$}, or simply {\it $(\cM,I)$-reducible} if and only if the composition $\rho_{I_0,I}\circ s_{0}$ factors via $q:\cU\to X$, that is, if there exists a morphism $\sigma_{I}:X\to Y_I$ such that the following diagram is commutative:
$$
\xymatrix@=35pt{\cU\ar[r]^{s_{0}}\ar[d]_{q}&Y_{I_0}\ar[d]^{\rho_{I_0,I}}\\X\ar[r]_{\sigma_{I}}&Y_I}
$$
Note that this condition implies that $\sigma_I$ is a section of $\pi_I$.

If $q$ has connected fibers, the condition can be restated by saying that $Y$ is $(\cM,I)$-reducible if and only if the cone $s_{0*}(\NE(\cU|X))$ lies in the extremal face $\NE(Y_{I_0}|Y_I)$. Then, given $Y$, either we may find a minimal subset $I_0\subseteq I\subsetneq  D $ such that $Y$ is $(\cM,I)$-reducible, 
or $Y$ is not $(\cM,I)$-reducible with respect to $I$, for every $I$. In this case we say that $\pi$ is {\it uniformly irreducible with respect to $\cM$}, or {\it $\cM$-irreducible}, for short. 
\end{definition}

\begin{lemma}\label{lem:minred}
If $\pi$ is $(\cM,J_1)$-reducible and $(\cM,J_2)$-reducible, then it is also $(\cM,J_1 \cap J_2)$-reducible.
\end{lemma}

\begin{proof}
Set $J=J_1\cap J_2$. By hypotheses, the maps $\rho_{I_0,J_i}:\cU\to Y_{J_i}$, $i=1,2$, factor via morphisms $\sigma_{J_i}:X\to Y_{J_i}$. For each $i$ let us set  $X'_i:=\rho_{J,J_i}^{-1}(\sigma_{J_i}(X))\subset Y_{J}$. One can easily check that on every fiber $\pi_{J}^{-1}(x)$, the intersection $X'_1\cap X'_2$ is a point;
then the map $X'_1\cap X'_2\to X$ is a bijection and, since $X$ is smooth, it is an isomorphism. Hence we have an inverse map $\sigma_{J}:X\to Y_J$, which satisfies $\rho_{I_0,J}\circ s_0=\sigma_J\circ q$, by construction. 
\end{proof}

In particular, when $\pi$ is $(\cM,I)$-reducible, the morphism $s_0$ factors via the fiber product $Y'_{I_0}:=Y_{I_0}\times_{Y_I}X$, which is a bundle over $X$, whose corresponding flag bundle is $\pi':Y':= Y\times_{Y_{I_0}}Y'_{I_0}=Y\times_{Y_I}X\to X$. We then have a commutative diagram:

$$
\xymatrix@=35pt{&Y'\pb\ar[d]\ar[r]&Y\ar[d]^{\rho_{I_0}}\\\cU\ar[r]\ar[rd]_q&Y'_{I_0}\pb\ar[r]\ar[d]&Y_{I_0}\ar[d]^{\rho_{I_0,I}}\\&X\ar[r]^{\sigma_{I}}\ar@{=}[rd]&Y_I\ar[d]^{\pi_I}\\&&X}
$$

Within the problem of finding diagonalizability conditions for uniform flag bundles, $(\cM,I)$-reducibility can be used to reduce the problem to flag bundles defined over groups of lower rank:

\begin{lemma}\label{lem:irred1}
With the same notation as above, $Y'$ is a uniform flag bundle over $X$, whose tag $\delta(Y')=(d'_i)_{i\in I}$ is a subtag of $\delta(Y)$, in the sense that $d'_i=d_i$ for all $I$ (considering $I$ as a subset of $ D $). Moreover, if $Y'$ is diagonalizable, then $Y$ is diagonalizable.
\end{lemma}

\begin{proof}
The first part is immediate by construction. For the second note that, by Proposition \ref{cor:decompo1}, the diagonalizability of $Y'$ is equivalent to the existence of a section of $Y'$ over $X$, which in turn provides a section of $Y$ over $X$. 
\end{proof}

\begin{remark}\label{rem:diagdis2}
Note that, in many cases, the Dynkin diagram of the flag $Y'\to X$ will be disconnected, and, according to Remark \ref{rem:diagdis}, the diagonalizability of $Y'$ will be reduced to the diagonalizability of a certain number of uniform flag bundles over $X$ associated with simple algebraic groups of smaller rank (one for each connected component of the Dynkin diagram of $Y'\to X$). 
\end{remark}

As a consequence of Lemma \ref{lem:irred1}, we have the following corollary:

\begin{corollary}\label{cor:irred2}
With the same notation as above, if $\pi:Y\to X$ is $(\cM,I_0)$-reducible, then it is diagonalizable.
\end{corollary}

\begin{proof}
Arguing as above, we consider the uniform bundle $Y'\to X$ whose tag is, in this case, equal to zero. We may then apply Theorem  
\ref{thm:trivial1} to claim that $Y'$ is trivial, hence it is diagonalizable and we may conclude by Lemma \ref{lem:irred1}.
\end{proof}


\subsection{Infinitesimal criteria for uniform reducibility}\label{ssec:GMcrit} 

Before starting, let us describe the set of hypotheses under which our results will work.

\begin{setup}\label{setup:5}
As in the previous section, we consider the case in which $X$ is a Fano manifold of Picard number one and $\pi:Y\to X$ is a $G/B$-bundle, uniform with respect to an unsplit dominating family of rational curves $\cM$, that we will assume to be {\it complete}, i.e. that $\cM$ is an irreducible component of the scheme $\rat^n(X)$. 
We will further assume that the evaluation morphism  $q:\cU\to X$ is a {\em contraction}, that is, it has connected fibers.
Finally, we will assume that $\pi$ is not trivial, equivalently, with the notation of the previous section, that $I_0 \subsetneq  D $ (cf. Theorem \ref{thm:trivial1}).
\end{setup} 

\begin{definition}\label{def:mindim}
Given a dominant projective morphism between irreducible varieties $g:M\to N$, we denote by $\dim(g)$ the relative dimension of $M$ over $N$, and we define its {\it contractibility dimension},  denoted by $\cdim(g)$, as the maximum integer $r$ satisfying that every morphism  $f:M\to M'$ over $N$ whose image has relative dimension smaller than $r$ is relatively constant.
Given an irreducible complex projective variety $M$, we define its contractibility dimension, denoted by $\cdim(M)$, as the contractibility dimension of the constant morphism. 
\end{definition} 

\begin{remark}
If  moreover $g:M\to N$ is a contraction, that is if it has connected fibers,  and $g':M'\to N$ is a surjective projective morphism satisfying that $\dim(g')<\cdim(g)$, then any morphism $f:M\to M'$ satisfying $g'\circ f=g$ factors via $g$, that is, there exists a morphism $\sigma:N\to M'$ such that $\sigma\circ g=f$. In particular, $\sigma$ is a section of $g'$:
$$
\xymatrix{M\ar[rd]_g\ar[rr]^f&&M'\ar@/^2mm/[dl]^{g'}\\&N\ar@/^2mm/[ur]^{\sigma}&}
$$ 
\end{remark}

\begin{remark}
The contractibility dimension of the evaluation $q:\cU\to X$ can be computed in many interesting examples  as the contractibility dimension of its general fiber. This is always the case if we assume the contraction $q$ to be  {\em quasi-elementary}, that is if, being $i:F \to \cU$ a general fiber of $q$, the image of $i_*:\Nu(F) \to \Nu(\cU)$ contains all the numerical classes of curves contracted by $q$ (see. \cite[Definition 3.1]{casagrande}). 
In fact, if this is the case, any morphism $s:\cU\to Z$ satisfying that the restriction to a general fiber $q^{-1}(x)$ is constant factors via $q:\cU\to X$. 

For instance, $q$ is quasi-elementary for the universal family of lines on a rational homogeneous manifold and, in the case the fibers of $q$  are homogeneous manifolds of the form $G/P$, the contractibility dimension of $q$ can be simply described as the minimum of the dimensions of the manifolds $G/P'$, where $P'\supset P$ is a parabolic subgroup containing $P$. To our best knowledge, studying which families of rational curves on Fano manifolds of Picard number one have quasi-elementary evaluation is an open problem. 
\end{remark}

Let us  consider now the morphism $s_0:\cU\to Y_{I_0}$ defined in Section  \ref{ssec:redu}, and the composition $\rho_{I_0,I}\circ s_0:\cU\to Y_I$, for $I\subset D $ a subset containing $I_0$. 

\begin{lemma}\label{lem:GM2} In the assumptions of \ref{setup:5} the bundle
$\pi:Y\to X$ is $(\cM,I)$-reducible if and only if, at the general smooth point $x$ of $\cU$  the composition
\begin{equation}\label{eq:differ}
\xymatrix@=35pt{{T_{\cU|X,x} } \ar[r]^(.45){ds_0} &\left(s_0^*T_{Y_{I_0}|X}\right)_x \ar[r]&\left(s_0^*\rho_{I_0,I}^*T_{Y_{I}|X}\right)_x}
\end{equation}
has rank smaller than $\cdim(q)$.
\end{lemma}

\begin{proof}
The condition on this rank being smaller than $\cdim(q)$ is equivalent to $\rank(d(\rho_{I_0,I}\circ s_0)_x)<\dim(X)+\cdim(q)$; by definition of contractibility dimension, this implies that $\rho_{I_0,I}\circ s_0$ is relatively constant over $X$; since the fibers of $q$ are connected by hypothesis, this implies that $\rho_{I_0,I}\circ s_0$ factors via $X$. This completes the proof of an implication, and its converse is obvious. 
\end{proof}

In the spirit of \cite[Proposition 3.2]{EHS}, rather than looking at the map (\ref{eq:differ}) at general points of a fiber $q^{-1}(x)$, we will look at its behaviour along a general fiber of $\cU$ over $\cM$,  
obtaining conditions on the tag of a uniform bundle for its reducibility or diagonalizability. More concretely, let $\rho_{I_0}:s_0^*Y=Y\times_{Y_{I_0}}\cU\to\cU$ be the pullback bundle, fitting in the diagram:
$$
\xymatrix@=35pt{s_0^*Y\ar[r]^{s_{0}}\ar[d]_{\rho_{I_0}}&Y\ar[d]^{\rho_{I_0}}\\\cU\ar[r]_{s_{0}}&Y_{I_0}}
$$
Let $\Gamma$ be a general 
fiber of $\cU$ over $\cM$, and $\ol{\Gamma}$ be any minimal section of $\rho_{I_0}$ over $\Gamma$ (note that, by Lemma \ref{lem:irred1}, $\Gamma\times_{Y_{I_0}}Y$ is trivial). Let us study the pullback map:
\begin{equation}\label{eq:diff}
\xymatrix@=35pt{\left(\rho_{I_0}^*T_{\cU|X}\right)_{|\ol{\Gamma}} \ar[r]^(.45){\rho_{I_0}^*ds_0} & \left(s_0^*\rho_{I_0}^*T_{Y_{I_0}|X}\right)_{|\ol{\Gamma}} \ar[r]& \left(s_0^*\rho_{I}^*T_{Y_{I}|X}\right)_{|\ol{\Gamma}} }.
\end{equation}
The completeness of the family $\cM$ allows us to claim that $$\left(\rho_{I_0}^*T_{\cU|X}\right)_{|\ol{\Gamma}}\cong\cO_{\ol{\Gamma}}(-1)^{\oplus\dim (q)}.$$  This in fact follows by the standard description of the differential morphism of the evaluation $q:\cU\to X$ (cf. \cite[II 3.4]{kollar}).

The splitting type of the target of (\ref{eq:diff}) may be controlled by taking an admissible ordering $\{L_1,\dots,L_m\}$ of $\Phi$ compatible with $I$,  
which provides 
a filtration (see Section \ref{ssec:filt}):
$$
0=\ol{\cE}_{m-k}\subset \ol{\cE}_{m-k-1}\subset\dots\subset\ol{\cE}_{0}=\rho_{I}^*T_{Y_{I}|X}
$$
with quotients: $\ol{\cE}_{r}/\ol\cE_{r+1}\simeq L_{m-r} \in \Phi^+ \setminus \Phi^+(I)$, for all $r$. Summing up we get:

\begin{proposition}\label{prop:MIred} Assume that the evaluation morphism $q:\cU\to X$ has contractibility dimension $m$, and that 
$$\#\{L_i \in \Phi^+ \setminus \Phi^+(I) ~|~L_i \cdot \ol{\Gamma} \le -1\} < m$$
Then $\pi:Y \to X$ is  $(\cM,I)$-reducible.
\end{proposition}

As a first application of Proposition \ref{prop:MIred} we obtain a flag bundle counterpart of the standard Grauert--M\"ulich theorem for vector bundles, that may be used, together with Lemma \ref{lem:irred1} and Remark \ref{rem:diagdis2}, in the problem of diagonalizability of low rank uniform flag bundles on Fano manifolds.

\begin{theorem}\label{thm:GMcrit1}
If $I_1:=\{i\in D \,|\,\,d_i\leq 1\}$ is a proper subset of $D$, then $\pi:Y\to X$ is $(\cM,I_1)$-reducible. 
\end{theorem}

\begin{proof} 
Since by hypothesis we have that $L_i\cdot\ol{\Gamma}\leq -2$ for all $L_i\in \Phi^+ \setminus \Phi^+(I_1)$, we conclude by Proposition \ref{prop:MIred}.
\end{proof}

We will now state the main result of this section, for which we need to introduce some notation. For every index $j\in I_1\setminus I_0$, that is, such that $d_j=1$, we denote by $\cD'_0(j)$ be the Dynkin subdiagram of $\cD$ supported on $I_0\cup\{j\}$, and by $\cD_0(j)$ the connected component of $\cD'_0(j)$ containing the node $j$. We denote by $I_0(j)$ the set of indices of $\cD_0(j)$, and by $m_j$ the number of positive roots $L$ of $G$ of the form:
$$
L=-\sum_{r\in I_0(j)}a_rK_r,\quad a_r\geq 0, \quad a_j=1.
$$

For the reader's convenience, we include here the values $m_j$ for every possible $j$, and every possible connected Dynkin diagram $\cD_0(j)$: 

\begin{table}[h]
\centering
\begin{tabular}{|c|c|c|}
\hline
$\cD_0(j)$     &j       &$m_j$           \\
\hline\hline
$\DA_n$   &$j$      &$j(n+1-j)$                 \\
$\DB_n$   &$j$  &$j(2n-2j+1)$               \\
$\DC_n$   &$(j<n,n)$    &$\left(j(2n-2j),\dfrac{n(n+1)}{2}\right)$                 \\
$\DD_n$   &$(j<n-2,n-2,n-1,n)$   &$\left(j(2n-2j),4(n-2),\dfrac{n(n-1)}{2},\dfrac{n(n-1)}{2}\right)$                 \\
$\DE_6$& $(1,2,3,4,5,6)$ &  $\big(16,20,20,18,20,16\big)$\\
$\DE_7$& $(1,2,3,4,5,6,7)$ &  $\big(32,35,30,24,30,32,27\big)$\\
$\DE_8$& $(1,2,3,4,5,6,7,8)$ &  $\big(64,56,42,30,40,48,54,56\big)$\\
$\DF_4$& $(1,2,3,4)$ &  $(14,12,6,8)$\\
$\DG_2$& $(1,2)$ &$(2,4)$\\

\hline
\end{tabular}
   \caption[Values of $m_j$ for connected Dynkin diagrams]{Values of $m_j$ for connected Dynkin diagrams} 
   \label{tab:mj}
\end{table}

\begin{theorem}\label{thm:some1}
Let $X$ be a Fano manifold, $\cM$ be an unsplit dominating complete family of rational curves, whose evaluation morphism $q:\cU\to X$ has connected fibers. Let $\pi:Y\to X$ be a uniform $G/B$-bundle over $X$, with tag $(d_1,\dots,d_n)$, and consider, for every node $j \in I_1\setminus I_0$, the integer $m_j$ defined above. 
If  $\cdim(q)>m_j$, for every $j\in I_1\setminus I_0$, then
$\pi$ is diagonalizable.
\end{theorem}

\begin{proof} We will show that $\pi$ is $(\cM,D \setminus \{j\})$-reducible for every $j\in I_1\setminus I_0$. Since $\pi$ is also $(\cM,I_1)$-reducible (Proposition  \ref{thm:GMcrit1}), it follows by Lemma \ref{lem:minred} that $\pi$ is $(\cM,I_0)$-reducible, hence diagonalizable by Corollary \ref{cor:irred2}. 

Fix an index $j\in I_1\setminus I_0$, and set $J:=D\setminus \{j\}$. Take an admissible ordering of $\Phi$ compatible with $I_0 \subsetneq J$ 
(see Definition \ref{def:admiscomp}), and the  
corresponding filtration of $\rho_{J}^*T_{Y_{J}|X}$,
whose quotients are isomorphic to classes  
$L_{m-r} \in \Phi^+ \setminus \Phi^+(J)$. Note that these are precisely the positive roots of $G$ containing $-K_j$ as a summand. All these classes have negative intersection with the minimal section $\ol{\Gamma}$, and, in order to apply Proposition \ref{prop:MIred}, we need to count those for which $L_{m-r}\cdot \ol{\Gamma}$ is equal to $-1$. This occurs only if $L_{m-r}$ belongs to the root subsystem determined by the Dynkin subdiagram $\cD'_0(j)$. Since this is the disjoint union of the root systems determined by the connected components of $\cD'_0(j)$, one such $L_{m-r}$ is necessarily a positive root for the connected Dynkin subdiagram $\cD_0(j)$, containing $-K_j$ as a summand with multiplicity one (being $-K_j\cdot\ol{\Gamma}=-1$).  
As there are $m_j<\cdim(q)$ of these classes $L_{m-r}$, we conclude that $\pi$ is $(\cM,J)$-reducible by Proposition \ref{prop:MIred}. 
\end{proof}

As a straightforward corollary, we remark that in the case $\cdim(q)>1 $,  the positivity of the tag implies diagonalizability. Note that the condition $\cdim(q)>1$ is obviously necessary, since the flag bundle determined by the universal bundle on any Grassmannian of lines is not diagonalizable, although it has tag equal to $(1)$. 

\begin{corollary}\label{prop:only1}
Let $X$ be a Fano manifold, $\cM$ be an unsplit dominating and complete family of rational curves, with evaluation morphism $q$ which has connected fibers and satisfies that $\cdim(q)>1$, 
Let $\pi:Y\to X$ be a uniform $G/B$-bundle over $X$. Then $\pi$ is diagonalizable unless $I_0\neq\emptyset$, that is, unless its tag contains a zero. 
\end{corollary}

Applied to uniform vector bundles, Corollary \ref{prop:only1} provides the following statement, that, in the case of $X\cong \P^n$, $n\geq 3$ was proven by Spindler in \cite{Spin}:

\begin{corollary}\label{cor:only1}
Let $X$ be a Fano manifold, $\cM$ be an unsplit dominating and complete family of rational curves, whose evaluation morphism $q$ has connected fibers and satisfies that $\cdim(q)>1$. Let $\cE$ be a vector bundle over $X$, uniform with respect to $\cM$, with splitting type $(a_1,\dots,a_r)$, $a_1<\dots<a_r$. Then $\cE$ is a direct sum of line bundles. 
\end{corollary}

\noindent \textit{Acknowledgements:} The authors would like to thank an anonymous referee for the helpful comments and questions, that greatly contributed to improving the final version of the paper. 

\bibliographystyle{plain}
\bibliography{biblio}

\end{document}